\documentclass[10pt, reqno]{amsart}
\usepackage{lipsum}
\usepackage{a4wide}
\usepackage{amssymb} 
\usepackage{amsmath}
\usepackage{amsthm} 
\usepackage{amsmath, amssymb, dsfont}
\usepackage{tikz-cd} 
\usepackage[all]{xy}
\usepackage{amscd}
\usepackage{hyperref}
\usepackage{enumitem}
\hypersetup{colorlinks,linkcolor={blue},citecolor={blue},urlcolor={black}}
\theoremstyle{plain}
\numberwithin{equation}{section}

\newcommand{\ind}{\operatorname{ind}}
\newcommand{\Ind}{\operatorname{Ind}}

\newcommand{\End}{\operatorname{End}}

\newtheorem{theorem}{Theorem}[subsection]
\newtheorem{corollary}[theorem]{Corollary} 
\newtheorem{lemma}[theorem]{Lemma}
\newtheorem{remark}[theorem]{Remark}

\newtheorem{definition}[theorem]{Definition}

\title{Asai Gamma Factors and Distinction in families}
\author{Sabyasachi Dhar and Hariom Sharma}
\date{\today}
\begin{document}
\begin{abstract}
Let $F$ be a finite extension of $\mathbb{Q}_p$ and let $E$ be a quadratic extension of $F$. A representation $(\pi,V)$ of ${\rm GL}_n(E)$ is said to be ${\rm GL}_n(F)$-distinguished if there exists a non-zero linear functional $\phi$ on $V$ such that $\phi(\pi(h)v) = \phi(v)$ for all $h \in {\rm GL}_n(F)$ and $v \in V$. In this article, we study the notion of ${\rm GL}_n(F)$-distinguished representations for $R[{\rm GL}_n(E)]$ modules of Whittaker type, where $R$ is a Noetherian algebra over the ring of Witt vectors of $\overline{\mathbb{F}}_\ell$ with $\ell \ne p$. We first derive a functional equation, which gives the existence of the Asai $\gamma$-factors associated with $R[{\rm GL}_n(E)]$ modules of Whittaker type. We then provide a necessary condition for cuspidal $R[{\rm GL}_n(E)]$ modules of Whittaker type to be Whittaker ${\rm GL}_n(F)$-distinguished, expressed in terms of their Asai $\gamma$-factors.
\end{abstract}
\maketitle
\section{Introduction}
Let $G$ be a group and let $H$ be a subgroup of $G$. A complex representation $(\pi,V)$ of $G$ is said to be $H$-distinguished if there is a non-zero linear form $\lambda:V\rightarrow \mathbb{C}$ which is $H$-equivariant, i.e.,  $\lambda(\pi(h)) v = \lambda(v)$
for $h\in H$ and $v\in V$. 
One of the fundamental problems in representation theory is to classify irreducible representations of $G$ that are $H$-distinguished, in terms of certain data attached to these representations. Usually, one considers this when $H$ is the subgroup of the fixed points of an involution of $G$. 
In this article, our discussion will focus on the situation where $G = {\rm GL}_n(E)$ and $H={\rm GL}_n(F)$ with $E$ being a quadratic extension of a $p$-adic field $F$, as considered in \cite{Flicker_distinguished, Dipendra_distinguished}. There is a characterization of discrete series representations of ${\rm GL}_n(E)$ being ${\rm GL}_n(F)$-distinguished, in terms of the poles of Asai $L$-factors attached to these representations, due to \cite[Theorem 4]{Kable_Asai} and \cite[Theorem 1.4]{Anand_Kable_Asai}. For various aspects of distinguished representations, especially in connection with Asai $L$-functions and Asai  $\gamma$-factors, we refer to \cite{Flicker_BSMF_1, Flicker_Asai, Kable_Asai, Anand_Kable_Asai, Anand_IMRN, Anand_IMRN_Asai, Nadir_IMRN}.

Recently, motivated by the study of congruences between automorphic representations, the representation theory over the field $\mathbb{C}$ has been extended to representations with coefficients in other fields (see \cite{Vigneras_modl_book}) and then further to more general rings (see \cite{LLC_families}), in which $p$ is invertible. In recent years, there has been a lot of interest in establishing several representation theoretic aspects over general rings, which are already known over $\mathbb{C}$; for a survey, we refer to \cite{Helm_Bernstein_Center, Helm_Whittaker, Helm_Moss_converse, DHMK_paper}. For $\ell$-modular cuspidal representations of ${\rm GL}_n(E)$ with $\ell \ne p$, ${\rm GL}_n(F)$-distinguishness has been studied by \cite{Nadir_Kurinczuk_mod_l_Asai, mod_l_distinguished}. However, distinguished representations over general rings are completely unexplored. The aim of this article is to study the ${\rm GL}_n(F)$-distinguished representations of ${\rm GL}_n(E)$ over a Noetherian $W(\overline{\mathbb{F}}_\ell)$-algebra $R$, where $W(\overline{\mathbb{F}}_\ell)$ is the ring of Witt vectors of $\overline{\mathbb{F}}_\ell$ with $\ell \not= p$. We focus on certain $R[{\rm GL}_n(E)]$ modules, namely Whittaker type, introduced by Emerton and Helm (\cite{LLC_families})--which is a generalization of generic representations over fields.

To state the main results of this article, some notations are in order. Let $\Lambda$ be the ring of Witt vectors of $\overline{\mathbb{F}}_\ell$, where $\ell$ is an odd prime different from $p$. Let $R$ be a Noetherian $\Lambda$-algebra. Let $V$ be a smooth $R[{\rm GL}_n(E)]$ module, where smooth means the stabilizer of every element in $V$ is an open subgroup of ${\rm GL}_n(E)$. Assume that $V$ is of Whittaker type (see $\S$\ref{def_co-Whittaker} for the definition). There is a Whittaker model associated with $V$, denoted $\mathbb{W}(V,\psi_E)$, which is a sub-$R[{\rm GL}_n(E)]$ module of ${\rm Ind}_{N_n(E)}^{{\rm GL}_n(E)}(\psi_E)$, where $N_n(E)$ is the group of unipotent upper triangular matrices in ${\rm GL}_n(E)$ and $\psi_E$ is a non-degenerate character of $N_n(E)$ that is trivial on $N_n(F)$.  We have a ${\rm GL}_n(E)$-equivariant surjection $V\rightarrow \mathbb{W}(V,\psi_E)$ but it is not necessarily an isomorphism, unlike in the classical setting. 

In the works of \cite{Kable_Asai, Anand_Kable_Asai}, they provide a characterization of discrete series complex representations of ${\rm GL}_n(E)$ being ${\rm GL}_n(F)$-distinguished, in terms of poles of Asai $L$-functions, constructed independently by Flicker (\cite{Flicker_Asai}) and Kable (\cite{Kable_Asai}) via establishing a functional equation for certain Zeta integrals similar to the Rankin--Selberg theory. In general, for a family of ${\rm GL}_n(F)$-representations of Whittaker type, the notion of $L$-function is not well-defined (see \cite[Section 1, P. 1790]{Moss_local_constants}). However, there is a notion of $\gamma$-factor associated with $R[{\rm GL}_n(F)]$ modules of Whittaker type, introduced by \cite{Moss_IMRN, Moss_local_constants, Moss_Nadir_Kirillov}--which are well behaved with respect to the base change of the coefficient rings. In this paper, we first derive a functional equation similar to \cite[Appendix, Theorem (iii)]{Flicker_Asai} for $R[{\rm GL}_n(E)]$ modules of Whittaker type, which gives us the existence of Asai $\gamma$-factors in families. Let us state the result more precisely.

For $W\in\mathbb{W}(V,\psi_E)$ and $\varphi\in C_c^\infty(F^n,R)$, we define a Laurent series with an indeterminate $X$, namely $I(X,W,\varphi)$ (see Section \ref{Asai_Zeta} for the definition), which belongs to the fraction ring $S^{-1}(R[X,X^{-1}])$, where $S$ is the multiplicative subset of $R[X,X^{-1}]$ consisting of Laurent polynomials whose first and last coefficients are units in $R$. Similarly, we have $I(X,\widetilde{W},\widehat{\varphi})$, where $\widetilde{W}(g) = W(w_n \,^tg^{-1})$ for $g\in {\rm GL}_n(E)$, and $\widehat{\varphi}$ is the Fourier transform with respect to a self-dual Haar measure on $F^n$. The first main result of this article is as follows.
\begin{theorem}\label{func_equ_intro}
Let $E$ be a quadratic extension of the $p$-adic field $F$. Let $R$ be a Noetherian $\Lambda$-algebra, and let $S$ be the multiplicative subset of $R[X,X^{-1}]$ consisting of Laurent polynomials whose first and last coefficients are units in $R$. Let $V$ be an $R[{\rm GL}_n(E)]$ module of Whittaker type. Then there exists a unique element $\gamma_{As}(X,V,\psi_E)\in S^{-1}(R[X,X^{-1}])$ such that 
$$ I(q^{-1}X^{-1},\widetilde{W},\widehat{\varphi}) = \omega_{V}(-1)^{n-1}\,\gamma_{As}(X,V,\psi_E)\,I(X,W,\varphi), $$
for all $W\in\mathbb{W}(V,\psi_E)$ and $\varphi\in C_c^\infty(F^n,R)$. Here, $\omega_V$ is the central character of $V$.
\end{theorem}
The element $\gamma_{As}(X,V,\psi_E)$ is called the Asai $\gamma$-factor associated with $V$, and it is compatible with the base change of the coefficient rings in the sense of \cite[Proposition 4.7]{Moss_local_constants}. For the proof of the above theorem, we closely follow the ideas of \cite{Flicker_Asai} and \cite{Moss_IMRN, Moss_Nadir_Kirillov}. When $E=F\times F$ and $R = A \otimes_{\Lambda} B$, where $A$ and $B$ are two Noetherian $\Lambda$-algebras, then Theorem \ref{func_equ_intro} gives a functional equation for a pair of $A[{\rm GL}_n(F)]$ and $B[{\rm GL}_n(F)]$ modules of Whittaker type, which completes the Rankin--Selberg theory of \cite[Theorem $5.3$]{Moss_Nadir_Kirillov} for $m=n$. 

For $R=\overline{\mathbb{F}}_\ell$, the notion of Asai $L$-functions for Whittaker type representations has been studied in \cite{Nadir_Kurinczuk_mod_l_Asai}, and they provide a characterization of $\ell$-modular ${\rm GL}_n(F)$-distinguished (in certain sense) irreducible cuspidal representations of ${\rm GL}_n(E)$ in terms of the poles of their Asai $L$-functions. The second main goal of this article is to show that for Whittaker type $R[{\rm GL}_n(E)]$ modules, a connection exists between the zeros of Asai $\gamma$-factor and ${\rm GL}_n(F)$-distinction; in fact, this connection is established for a more stronger notion of distinction, which we call Whittaker distinguished.

An $R[{\rm GL}_n(E)]$ module of Whittaker type $V$ is said to be Whittaker ${\rm GL}_n(F)$-distinguished if there exists a non-zero ${\rm GL}_n(F)$-equivariant linear form $\phi:\mathbb{W}(V,\psi_E)\rightarrow R$. By precomposing $\phi$ with the canonical map $V\rightarrow \mathbb{W}(V,\psi_E)$, the module $V$ becomes ${\rm GL}_n(F)$-distinguished. When $R$ is a field and $V$ is an irreducible generic representation of ${\rm GL}_n(E)$ that is ${\rm GL}_n(F)$-distinguished, then it is also Whittaker ${\rm GL}_n(F)$-distinguished because the canonical surjection $V\rightarrow \mathbb{W}(V,\psi_E)$, in this case, is an isomorphism. In general, there is no guarantee that a ${\rm GL}_n(F)$-distinguished $R[{\rm GL}_n(E)]$ module of Whittaker type is also Whittaker ${\rm GL}_n(F)$-distinguished. Therefore, the notion of Whittaker distinguished is stronger than the usual notion of distinction for Whittaker type representations in families.
We now state the second main result of this article, which is as follows.
\begin{theorem}\label{main_thm_intro}
Let $E$ be a quadratic extension of a $p$-adic field $F$. Let $V$ be a cuspidal $R[{\rm GL}_n(E)]$ module of Whittaker type such that $V$ is Whittaker ${\rm GL}_n(F)$-distinguished. Then the Asai $\gamma$-factor $\gamma_{As}(X,V,\psi_E)$ vanishes at $X=1$.
\end{theorem}
The key idea of the proof of Theorem \ref{main_thm_intro} is similar to that of \cite[Theorem 1.4]{Anand_Kable_Asai}, i.e., using the functional equation of Theorem \ref{func_equ_intro}. Note that the proof of Theorem \ref{main_thm_intro} relies on certain integrals of Whittaker functions (see $\S$ $6.1$), which are well-defined when $V$ is cuspidal. 

The paper is organized as follows. In Section $2$, we discuss some preliminaries that are used in this paper. In Section $3$, we discuss zeta integrals and present some initial results. Section $4$ is devoted to the proof of the rationality of certain Laurent series with zeta integrals as its coefficients. In Section $5$, we prove Theorem \ref{func_equ_intro}. In Section $6$, we discuss the distinguished representations for the Whittaker type families and prove Theorem \ref{main_thm_intro}.

\section{Preliminaries}
In this section, we introduce some notation and recall definitions and standard results used throughout the paper. For this part, we mainly follow the expositions of \cite{BZ_2, Rankin_Selberg_mod_l, Helm_Whittaker, Moss_local_constants}.

Let $F$ be a finite extension of $\mathbb{Q}_p$ with $\mathfrak{o}_F$ the ring of integers of $F$, and $\mathfrak{p}_F$ denote the unique maximal ideal of $\mathfrak{o}_F$. We fix a uniformizer $\varpi_F$ of $F$. Let $k_F$ be the finite residue field $\mathfrak{o}_F/\mathfrak{p}_F$, and let $q$ be the cardinality of $k_F$. We denote by $G_n$ the group ${\rm GL}_n(F)$. Let $\Lambda$ be the ring of Witt vectors of $\overline{\mathbb{F}}_\ell$, where $\ell$ and $p$ are distinct primes. We further assume that $\ell$ is odd, so that the ring $\Lambda$ contains a square root of $q$ (When $\ell = 2$, all arguments presented here will remain valid after possibly adjoining a square root of $q$ to $\Lambda$). Let $R$ be a commutative ring with unity. In this article, all $R[G_n]$ modules are assumed to be smooth, i.e., the stabilizer of each element in the module is an open subgroup of $G_n$. A smooth $R[G_n]$ module $V$ is called admissible if for every compact open subgroup $K$ of $G_n$, the space of $K$-fixed points $V^K$ is a finitely generated $R$-module. Given a prime ideal $\mathfrak{P}$ of $R$, the fiber $V\otimes_R \kappa(\mathfrak{P})$ is denoted by $V_\mathfrak{P}$, where $\kappa(\mathfrak{P})$ is the field of fraction of $R/\mathfrak{P}$. In this article, we are interested in a special class of smooth $R[G_n]$ modules, where $R$ is a Noetherian $\Lambda$-algebra, in which case the fibers admit a unique absolutely irreducible generic quotient. These are the natural generalizations of generic representations of $G_n$ over fields in this context.

If $V$ is a smooth $R[H]$ module for a locally profinite group $H$, and 
$\theta : H \to R^{\times}$ is a smooth character, we define
\(
V_{H,\theta} = V / V(H,\theta),
\)
where $V(H,\theta)$ is the sub-$R$-module generated by the elements of the form 
$h v - \theta(h)v$ for $h \in H$ and $v \in V$. 
For a closed subgroup $K \subseteq H$ and a smooth $R[K]$ module $W$, the induced 
representation $\Ind_{K}^{H}W$ consists of all functions 
$f : H \to W$ satisfying the relation 
\[
f(kh) = k \cdot f(h)
\]
for every $k \in K$ and $h \in H$, and which are invariant under right translation 
by some compact open subgroup of $H$. The compact induction 
$\ind_{K}^{H}W$ is the subspace of such functions that are compactly 
supported modulo $K$. We denote by $I_H^G$ (resp. $\iota_H^G$) the normalized induction (resp. compact induction) functor.

For any standard parabolic subgroup $P$ of $G_n$, we have a decomposition 
$P = MU$, where $M$ is the standard Levi subgroup (block diagonal matrices) 
and $U$ is the unipotent radical (strictly block upper triangular matrices). 
The Jacquet functor associated with $M$, denoted $J_M$, sends a smooth $R[G_n]$ module
$V$ to the quotient $V_{M,1}$ after restricting $V$ to $P$. Moreover, the smooth $R[G_n]$ module 
$V$ is called cuspidal if its Jacquet module $J_M(V)$ is zero 
for every proper Levi subgroup $M$ of $G_n$. The Jacquet functor $J_L$ admits a right adjoint, namely parabolic induction. 
Starting with a smooth $R[M]$ module $W$ for a Levi subgroup $M$, one first inflates 
$V$ to a $R[P]$ module by letting the 
unipotent part $U$ act trivially on $W$ and then takes the compact induction 
$\ind_{P}^{G_n} W$.

\subsection{Whittaker module}\label{def_co-Whittaker}
Let $\psi: F\rightarrow \Lambda^\times$ be a non-trivial additive character. Let $N_n$ be the group of unipotent upper triangular matrices in $G_n$. Then the character $\psi$ induces a non-degenerate character of $N_n$, also denoted by $\psi$, and is defined by
\begin{equation}\label{non_deg_char}
\psi\big((x_{ij})_{i,j=1}^n\big) := \psi(x_{12}+ x_{23} + \cdots+ x_{(n-1)n}).    
\end{equation}
Let $R$ be a Noetherian $\Lambda$-algebra. By abuse of notation, the composition $F\xrightarrow{\psi} \Lambda^\times\rightarrow R^\times$ is also denoted by $\psi$. 
\begin{definition}{\rm (\cite[Definition 2.4]{Moss_Nadir_Kirillov})}
A smooth $R[G_n]$ module $V$ is said to be of {\it Whittaker type} if $V$ is admissible, is finitely generated as $R[G_n]$-module, and the twisted Jacquet module $V_{N_n,\psi}$ is a free $R$-module of rank one.     
\end{definition}
If $V$ is a smooth $R[G_n]$ module of Whittaker type, then we have an isomorphism (\cite[Proposition 3.1]{Moss_Nadir_Kirillov})
$$ {\rm End}_{R[G_n]}(V)\xrightarrow{\sim} R. $$
In particular, the center of $G_n$ acts on $V$ via a character, denoted by $\varpi_V:F^\times\rightarrow R^\times$. In other words, every $R[G_n]$ module of Whittaker type admits a central character. Since $V_{N_n,\psi}\simeq R$, it follows from \cite[Chapter III, Section 1.8]{Vigneras_modl_book} that
$$ {\rm Hom}_{N_n}(V,\psi) \simeq {\rm Hom}_{G_n}(V,{\rm Ind}_{N_n}^{G_n}\psi). $$
Let $\varphi$ be a generator of the Hom-space ${\rm Hom}_{N_n}(V,\psi)$. Then the image of $\varphi$, under the above isomorphism, is the $G_n$-equivariant $R$-linear map $v\mapsto W_v$, where $W_v:G_n\rightarrow R$ is defined as $W_v(g):=\varphi(gv)$ for $g\in G_n$. The collection $\{W_v:v\in V\}$ is called the Whittaker space of $V$, denoted by $\mathbb{W}(V,\psi)$. Thus, we have a $G_n$ equivariant surjection $V\rightarrow \mathbb{W}(V,\psi)$, but it is not necessarily an isomorphism. Therefore, different $R[G_n]$ modules of Whittaker type can have the same Whittaker space (see \cite[Section 2.2]{Moss_local_constants}). 

For an $R[G_n]$ module $V$ of Whittaker type, we denote by $V^\vee$ the smooth $R[G_n]$ module with the same underlying set as $V$, but for which the action of $g$ on $V^\vee$ coincides with the action of $w_n\,^tg^{-1}w_n$ on $V$, where $w_n$ is matrix in $G_n(\mathbb{Z})$ with $1$'s on the anti-diagonal entries and $0$'s elsewhere. Note that $V^\vee$ is admissible, $R[G_n]$-finitely generated, and $V^\vee_{N_n,\psi^{-1}}$ is a free $R$-module of rank $1$. Thus, $V^\vee$ is also an $R[G_n]$ module of Whittaker type. For $W\in \mathbb{W}(V,\psi)$, let $\widetilde{W}$ be the function on $G_n$, defined by $\widetilde{W}(g) = W(w_n\,^tg^{-1})$ for $g\in G_n$. Then the map $W\mapsto \widetilde{W}$ is an $R$-module isomorphism from $\mathbb{W}(V,\psi)$ onto $\mathbb{W}(V^\vee,\psi^{-1})$ (see \cite[Section 2, P. 1005]{Helm_Moss_converse} for details).
\subsection{Derivatives}
In this subsection, we recall the notion of Bernstein--Zelevinsky derivatives as in \cite[Section 3]{BZ_2}, which also carry over to this setting (\cite[Definition 3.1.1]{LLC_families}). Let $P_n(F)$ be the mirabolic subgroup of $G_n(F)$. We embed the group $G_{n-1}(F)$ into $P_n(F)$ via the map $g\mapsto {\rm diag}(g,1)$. Let $U_n(F)$ be the subgroup of $P_n(F)$, defined as 
$$ U_n(F) := \bigg\{\begin{pmatrix}
    I_{n-1} & x\\
    0 & 1
\end{pmatrix} \mid x\in F^{n-1}\bigg\}. $$
Then $P_n(F)$ is the semi-direct product of $G_{n-1}(F)$ with $U_n(F)$. For a locally profinite group $G$ and a Noetherian $\Lambda$-algebra $R$, we denote by $\mathcal{M}(G)$ the category of smooth $A[G]$ modules. We have the following exact functors
$$ \Phi^{+} : \mathcal{M}(P_{n-1}(F)) \longrightarrow \mathcal{M}(P_n(F)),\,\,\,\, \Phi^{-} :\mathcal{M}(P_{n}(F)) \longrightarrow \mathcal{M}(P_{n-1}(F)), $$
$$ \Psi^{+} :\mathcal{M}(G_{n-1}(F)) \longrightarrow \mathcal{M}_A(P_n(F)),\,\,\,\, \Psi^{-} : \mathcal{M}(P_{n}(F)) \longrightarrow \mathcal{M}(G_{n-1}(F)), $$
defined by 
$$\Phi^+(V) = {\rm ind}_{P_{n-1}U_n(F)}^{P_n(F)}(V\otimes \psi),\,\,\,\, \Phi^-(M) = M_{U_n(F),\psi}, $$
$$\Psi^+(N)= \delta_{U_n(F)}^{1/2}N\otimes\psi,\,\,\,\, \Psi^-(M) = M_{U_n(F),\mathds{1}} $$ 
for $V\in \mathcal{M}(P_{n-1}(F))$, $M\in\mathcal{M}(P_n(F))$, and $N\in \mathcal{M}(G_{n-1}(F))$. The $i$-th Bernstein--Zelevinsky derivative functor $(-)^{(i)} : \mathcal{M}(P_n(F))\rightarrow\mathcal{M}(G_{n-i}(F))$ is defined as 
$$ M^{(i)} = \Psi^-(\Phi^-)^{i-1} M $$
for $M\in \mathcal{M}(P_n(F))$. For a smooth $R[G_n(F)]$ module $V$, the $i$-th derivative $V^{(i)}$ is defined as $V^{(i)} = (V|_{P_n(F)})^{(i)}$. There is a filtration of $R[P_n(F)]$ submodules associated with $V$, namely
$$ V=V_1\supseteq V_2\supseteq\cdots\supseteq V_{n}\supseteq \{0\} $$
with $V_i \simeq (\Phi^+)^{i-1}(\Phi^-)^{i-1}(V)$ and $V_i/V_{i+1} \simeq (\Phi^+)^{i-1}\Psi^+(V^{(i)})$. If $V$ is of Whittaker type, then $V_n \simeq (\Phi^+)^{n-1}\Psi^+(\mathds{1})$, which equal to the compact induction ${\rm ind}_{N_n(F)}^{P_n(F)}\psi$.

\section{Asai zeta integrals and formal series}\label{Asai_Zeta}
In this section, we introduce some further notation regarding Asai zeta integrals and set up initial results to prove Theorem \ref{func_equ_intro}. We closely follow the analogous Rankin--Selberg theory in families due to \cite{Moss_IMRN, Moss_local_constants}.

Let $F$ be a finite extension of $\mathbb{Q}_p$, and let $E$ be a quadratic extension of $F$. Fix an element $\delta\in E\setminus F$ such that $\delta^2\in F$, hence $E=F(\delta)$. Let $\sigma$ be the non-trivial Galois automorphism. We use the notations $G_n(F)$ and $G_n(E)$ for the general linear group $G_n$ to mention the ground field explicitly. Let $\psi_E$ be a non-trivial additive character of $E$ trivial on $F$, it is of the form $x\mapsto \psi({\rm Tr}_{E/F}(\delta x))$ for some non-trivial additive character $\psi$ of $F$, where ${\rm Tr}_{E/F}$ is the trace function. We also denote by $\psi$ (resp. $\psi_E$) the non-degenerate character of $N_n(F)$ (resp. $N_n(E)$), defined as (\ref{non_deg_char}). Let $R$ be a Noetherian $\Lambda$-algebra. Let $C_c^\infty(F^n, R)$ denote the space of locally constant and compactly supported $R$-valued functions on $F^n$. For $\varphi\in C_c^\infty(F^n,R)$, we denote by $\widehat{\varphi}$ the Fourier transformation of $\varphi$ with respect to the $\psi$-self-dual $R$-Haar measure $dx$ on $F^n$ satisfying $dx(\mathfrak{o}_F) = q^{-m/2}$, where $m$ is the smallest integer such that $\psi|_{\mathfrak{p}_F^m}$ is trivial.

For every $r\in \mathbb{Z}$, we set $G_n(F)^r = \{g\in G_n(F) : v({\rm det}g) = r\}$. More generally, $X^r := X\cap G_n(F)^r$ for any $X\subseteq G_n(F)$. Let $\eta = (0,\dots,0,1)\in F^n$. Let $V$ be an $R[G_n(E)]$ module of Whittaker type. For $W\in \mathbb{W}(V,\psi_E)$ and $\varphi\in C_c^\infty(F^n,R)$, consider the following integrals
$$ c_r(W, \varphi) := \int_{N_n(F)\backslash G_n(F)^r} W(g)\,\varphi(\eta g)\, dg,\,\,\,\, r\in\mathbb{Z}, $$
where $dg$ is the right $G_n(F)$-invariant Haar measure on the coset space $N_n(F)\backslash G_n(F)$. 
\begin{lemma}
Let $\varphi\in {\rm Ind}_{N_n(E)}^{G_n(E)}\psi_E$. Then, for every $m\in\mathbb{Z}$, the restriction $\varphi|_{G_n(E)^m}$ is compactly supported modulo $N_n(E)$.
\end{lemma}
\begin{proof}
In view of Iwasawa decomposition, it is sufficient to prove that $\varphi|_{T_n(E)^m}$ is compactly supported. However, because of the restriction on the determinant in $T_n(E)^m$, it is sufficient to show that there is $r\in \mathbb{Z}$ such that $W({\rm diag}(a_1,\dots, a_n)) = 0$ whenever $v(a_i)- v(a_{i+1})\leq r$ for $i=1,2,\dots, n-1$. For $x\in E$, consider the matrix 
$$ u_i(x) = 
\begin{pmatrix}
1 & & & & & &\\
& . & & & & &\\
& & 1 & & & &\\
& & & 1 & x &\\
& & & 0 & 1 &\\
& & & & & 1 &\\
& & & & & & .\\
& & & & & & & 1
\end{pmatrix} \in G_n(E). $$
For $x$ sufficiently small, $\varphi$ is fixed by $u_i(x)$ on the right. Then 
$$ (u_i(x).\varphi -\varphi)({\rm diag}(a_1,\dots,a_n)) = \Big(\psi_E\big(\frac{a_ix}{a_{i+1}}\big)-1\Big)\,\varphi({\rm diag}(a_1,\dots,a_n)). $$
Now, the element $\psi_E\big(\frac{a_ix}{a_{i+1}}\big)-1$ is invertible in $\Lambda$ for $v(a_i)-v(a_{i+1})$ small enough, and hence invertible in $R$. Thus, we get $\varphi({\rm diag}(a_1,\dots,a_n)) = 0$ and the result follows.
\end{proof}
Therefore, the integrals $c_r(W,\varphi)$, $r\in\mathbb{Z}$, are well-defined. Now, we define the formal power series
$$ I(X,W,\varphi) = \sum_{r\in \mathbb{Z}} c_r(W,\varphi) X^r \in R[[X,X^{-1}]]. $$
\begin{lemma}\label{285}
The series $I(X,W,\varphi)$ has finitely many non-zero positive powers of $X^{-1}$, and thus is an element of $R[[X]][X^{-1}]$.
\end{lemma}
\begin{proof}
By Iwasawa decomposition $G_n(F) = B_n(F)K_0(F)$, where $B_n(F)$ is the group of upper triangular matrices in $G_n(F)$ and $K_0(F)$ is the maximal compact open subgroup of $G_n(F)$, it reduces to consider the integral 
$$ \int_{N_n(F)\backslash T_n(F)^r} W(a)\,\varphi(\eta a)\, da, $$
where $T_n(F)$ is the diagonal torus in $G_n(F)$. We now show that there exists an integer $N>0$ such that  $W(a)\,\varphi(\eta a)$ vanishes for all $a\in T_n(F)$ with $v({\rm det}(a)) < N$. Choose an integer $j$ such that $W$ is stabilized by the compact open subgroup 
$$ \begin{pmatrix}
    1 & \mathfrak{p}_F^j &  & \\
      & 1 &  & \\
      &  & . &   \\
      &  &  & . \\
      &  &  &  & 1    
\end{pmatrix} $$
Let $x\in \mathfrak{p}_F^j$. Then we have 
$$ (1-\psi(a_1x))\,W(a)\,\varphi(\eta a) = 0, $$
where $a_1$ is the first entry in the diagonal matrix $a$. Whenever $v({\rm det}(a))$ is negetive enough that $a_1x$ lies outside ${\rm ker}(\psi)$, we get that $\psi(a_1x)$ is non-trivial $p$-power root of unity in $\Lambda$. Then $1-\psi(a_1x)$ is a lift of some non-zero element in $\overline{\mathbb{F}}_\ell$, and hence is a unit in $\Lambda$. Thus, for such $a\in T_n(F)$, we get 
$$ W(a)\,\varphi(\eta a) = 0. $$
Hence the lemma.
\end{proof}

\section{Rationality of Asai Formal Series}\label{rationality}
This section is devoted to the proof of a rationality result of the formal series $I(X,W,\varphi)$, introduced in Section \ref{Asai_Zeta}. We first prove some essential lemmas. Let $V$ be a smooth $R[G_n(E)]$ module of Whittaker type, where $R$ is a Noetherian $\Lambda$-algebra. For simplicity of notation, let
$\mathcal{W}$ be the Whittaker space $\mathbb{W}(V,\psi_E)$. The restriction to $T_n(F)$ map
\[\mathcal{W} \longrightarrow C^\infty(T_n(F),R),
\]
gives a surjective map of $R$-modules. We denote by $\mathcal{V}$ the image of this restriction map. Let $v : F \longrightarrow \mathbb{Z}$
denote the valuation associated with \( F \).  
For a function \( \varphi \) defined on \( T_n(F) \), we say that
$\varphi(a) \to 0 \quad \text{uniformly as } v(a_i) \to \infty$
if there exists a constant \( N > 0 \) such that whenever \( v(a_i) \geq N \), the function vanishes at $a$, i.e., $\varphi(a) = 0$.
Following \cite[Section $3$, P. 4912]{Moss_IMRN}, we define the set
\[
\mathcal{V}_i := \left\{ \varphi \in \mathcal{V} \ \middle|\ \varphi(a) \to 0 \ \text{uniformly as } v(a_i) \to \infty \right\}.
\]
For each \(1\leq i \leq n \), let \( M_n(i) \) be the standard Levi subgroup $G_i(E) \times G_{n-i}(E)$
and denote by \( N_n(i) \) its unipotent radical. The lemma below is proved using the similar idea as in \cite[Lemma $3.3$]{Moss_IMRN}.
\begin{lemma}\label{100}
For each $i$, let $\theta_i$ be the composition $\mathcal{W} \rightarrow \mathcal{V} \rightarrow \mathcal{V}/\mathcal{V}_i$. Then the submodule $\mathcal{W}(N_n(i), \mathds{1})$ of $\mathcal{W}$  is contained in $\ker (\theta_i)$.
\end{lemma}
\begin{proof}
By definition, the submodule $\mathcal{W}(N_n(i), \mathds{1})$ is generated by elements of the form 
$u.\varphi - \varphi$, where $u\in N_n(i)$ and $\varphi \in \mathcal{W}$.  
For any $g \in G_n(E)$ and $x\in N_n(E)$, we have $\varphi(xg) = \psi_E(x)\,\varphi(g)$. 
Now taking 
\[
a =
\begin{pmatrix}
a_1 & \cdots & a_n & \\
    &        &     & a_2 & \cdots & a_n \\
 & & &  & &  &\ddots &  \\
 & & & & & & & a_n
\end{pmatrix}
\in T_{n}(F) \quad \text{and}
\quad
u =
\begin{pmatrix}
I_i & Y \\
0 & I_{n-i}
\end{pmatrix}
\in N_n(i)
\]
we obtain
\begin{align*}
(u.\varphi)(a) 
&= \varphi\bigl( a u a^{-1} a\bigr) \\
&= \psi_E\bigl( a u a^{-1} \bigr)\, \varphi(a) \\
&= \psi_E\!
\begin{pmatrix}
I_i & z \\
0 & I_{n-i}
\end{pmatrix}\,
\varphi(a),
\end{align*}
where $z$ is an $i \times (n-i)$ matrix whose bottom left entry is $a_i Y_{i,1}$. Then this expression equals $\psi_E (a_i Y_{i,1})\,\varphi(a)$. Since $\ker(\psi_E)$ is open, for $\upsilon(a_i)$ sufficiently large, we have $\psi_E(a_i Y_{i,1}) = 1$, and this shows that  
$(u.\varphi)(a)- \varphi(a) = 0$ for $v(a_i)$ sufficiently large. Hence, the lemma.
\end{proof}
As a consequence, we have
\begin{corollary}
The map
$\theta_i : \mathcal{W} \longrightarrow \mathcal{V}/\mathcal{V}_i$,
defined as in the construction preceding Lemma \ref{100}, factors through the Jacquet module $\ J_{M_n(i)}\mathcal{W}$.
\end{corollary}
\begin{proof}
By Lemma \ref{100}, the map \( \theta_i \) is trivial on the $R$-submodule
$\mathcal{W}(N_n(i), \mathds{1})$ of $\mathcal{W}$.
Hence, it factors through the quotient
$J_{M_n(i)}\mathcal{W}$.
\end{proof}
\subsection{}
Note that the torus \( T_n(F) \) can be parametrized by
$(a_1,\dots,a_n) \in (F^\times)^n$
via the map
\[
(a_1,\dots,a_n) \ \longmapsto\ 
a =
\begin{pmatrix}
a_1\cdots a_n & & & \\
& a_2\cdots a_n & & \\
& & \ddots & \\
& & & a_n
\end{pmatrix}
\in T_n(F),
\]
We now define the operator \( \rho_i(\varpi_F) \) as right translation by the diagonal matrix
\[
D_i(\varpi_F) := \mathrm{diag}(\underbrace{\varpi_F,\dots,\varpi_F}_{i\ \text{times}},
\underbrace{1,\dots,1}_{n-i\ \text{times}}).
\]
In other words, \( \rho_i(\varpi_F) \) acts on functions on \( T_n(F) \) by precomposing with the map
\[
a \ \longmapsto\ a D_i(\varpi_F).
\]
Since both \( a \) and \( D_i(\varpi_F) \) are diagonal, their product is diagonal with entries multiplied coordinatewise. Writing the diagonal of \( a \) as
\[
(a_1\cdots a_n,\ a_2\cdots a_n,\ \dots,\ a_i\cdots a_n,\ a_{i+1}\cdots a_n,\ \dots,\ a_n),
\]
the diagonal of \( a D_i(\varpi_F) \) is
\[
(\varpi_F(a_1\cdots a_n),\ \varpi_F(a_2\cdots a_n),\ \dots,\ \varpi_F(a_i\cdots a_n),\ 
a_{i+1}\cdots a_n,\ \dots,\ a_n).
\]
In terms of the parameter \( (a_1,\dots,a_n) \), this corresponds exactly to multiplying
the \( i \)-th parameter by \( \varpi_F \) while leaving all other \( a_j \) unchanged:
\[
(a_1,\dots,a_{i-1},a_i,a_{i+1},\dots,a_n)
\ \longmapsto\
(a_1,\dots,a_{i-1},a_i\varpi_F,a_{i+1},\dots,a_n).
\]
Therefore, the right translation by \( D_i(\varpi_F) \) acts on a function \( \varphi \) on \( T_n(F) \) as
\[
(\rho_i(\varpi_F)\varphi)(a)
= \varphi(a D_i(\varpi_F))
= \varphi(a_1,\dots,a_{i-1},a_i\varpi_F,a_{i+1},\dots,a_n).
\]
Note that the submodule $\mathcal{V}_i$ is contained in ${\rm ker}(\rho_i(\varpi_F))$. Therefore, we consider the operator 
$\rho_i(\varpi_F)$
on the quotient space \( \mathcal{V} / \mathcal{V}_i \) defined by right translation by the diagonal matrix $D_i(\varpi_F)$. We now present the following lemma, whose proof follows the technique used in \cite[Lemma $3.5$]{Moss_IMRN}.
\begin{lemma}\label{finite_gen}
Let \( B_i \) be the \( R \)-subalgebra of 
$\mathrm{End}_R(\mathcal{V} / \mathcal{V}_i)$
generated by \( \rho_i(\varpi_F) \).  
Then \( B_i \) is finitely generated as an \( R \)-module.
\end{lemma}
\begin{proof}
For each index \( i \), let \(\rho_i(\varpi_F)\) denote the right translation operator induced by the diagonal element $D_i(\varpi_F)$
acting on Whittaker functions. Since \( D_i(\varpi_F) \) normalizes the unipotent subgroup \( N_n(i) \), the operator \(\rho_i(\varpi_F)\) preserves the subspace defining the \((N_n(i),\mathds{1})\)-coinvariants. Consequently, \(\rho_i(\varpi_F)\) induces an \( R[M_n(i)] \)-linear endomorphism of the Jacquet module
\[
\rho_i(\varpi_F): J_{M_n(i)}\mathcal{W} \longrightarrow J_{M_n(i)}\mathcal{W},
\]
Consider the natural surjection
\[
\pi: J_{M_n(i)}\mathcal{W} \twoheadrightarrow\ \mathcal{V}/\mathcal{V}_i
\]
arising from the factorization of \(\theta_i\). Because \(\rho_i(\varpi_F)\) corresponds to right translation by \( g_i \) and \( g_i \) normalizes the unipotent radical, we have
\[
\rho_i(\varpi_F)\bigl(\ker \pi\bigr)\subseteq \ker \pi.
\]
Thus, \(\rho_i(\varpi_F)\) descends to a well-defined endomorphism of \(\mathcal{V}/\mathcal{V}_i \). It follows that the $R$-subalgebra of
\(\End_R(\mathcal{V}/\mathcal{V}_i)\) generated by \(\rho_i(\varpi_F)\) coincides with the image of the subalgebra of
\[
\End_R\bigl(J_{M_n(i)}\mathcal{W}\bigr)
\]
generated by the same operator.
Using the results of \cite[Proposition 9.12, Proposition 9.15]{Helm_Bernstein_Center} and \cite[Lemma 1]{Bushnell_Hecke}, 
we get that the Jacquet functor $J_M$ for any parabolic subgroup $P=MN$ of $G_n$ with Levi component $M$, preserves admissibility when applied to primitive admissible \( R[G_n] \) modules. 
Since \( V \) is both admissible and finitely generated \( R[G_n] \) module, it belongs to only finitely many Bernstein components. Equivalently, there are only finitely many primitive orthogonal idempotents \( e \) in the Bernstein center with 
\( eV \neq 0 \). 
In particular, applying \( J_{M_n(i)} \) to the Whittaker space $\mathcal{W}$, we get that
\(J_{M_n(i)}\mathcal{W}\)
is also an admissible \( R[M_n(i)] \) module. By a classical result of Bernstein--Zelevinsky (\cite[Proposition~$3.13$(e)]{BZ_2}), 
this Jacquet module is also of finite type over \( R[M_n(i)] \). 
The essential geometric input is that for the parabolic subgroup
\(P_n(i)=M_n(i)N_n(i),\)
the quotient \( P_n(i)\backslash G \) is compact, which ensures that the Jacquet module $J_{M_n(i)}\mathcal{W}$ can be generated by finitely many elements 
as an \( R[M_n(i)] \) module.

Let us choose a finite generating set \( \{w_\alpha\} \) for $J_{M_n(i)}\mathcal{W}$. By smoothness of $V$, each \( w_j \) is stabilized by a compact open subgroup in \( M_n(i) \), say $U_\alpha$. Set $U=\bigcap_{\alpha} U_\alpha$. Thus, we get a compact open subgroup $U$ 
in $M_n(i)$ that fixes each generator $w_\alpha$ simultaneously. Now, any \( R[M_n(i)] \) equivariant endomorphism of the Jacquet module $J_{M_n(i)}\mathcal{W}$ is fully determined 
by its action on these finitely many generators. Since the endomorphisms are \( M_n(i) \)-equivariant 
and \( U \) fixes each \( w_\alpha \), the images of these generators must lie inside the 
\( U \)-invariant subspace
\[
(J_{M_n(i)}\mathcal{W})^U.
\]
Since \( V \) is admissible, this \( U \)-invariant subspace is a finitely generated \( R \)-module. 
As a result, each endomorphism corresponds to finitely many elements in this finite \( R \)-module, and therefore the endomorphism ring
\[
\End_{R[M_n(i)]}\big(J_{M_n(i)}\mathcal{W}\big)
\]
is itself finitely generated as an \( R \)-module. Since the subalgebra \( B_i \) generated by the operator 
\( \rho_i(\varpi_F) \) embeds into this endomorphism ring and $R$ is Noetherian, the $R$-subalgebra \( B_i \) is also finitely generated as an \( R \)-module.
\end{proof}
\subsection{}
For each $1 \leq j \leq n$, let $\mathcal{V}^j$ (resp. $\mathcal{V}_i^j$) denote the $R$-submodule of $C^\infty(T_j(F),R)$ (see \cite[Section $3$, P. 4914]{Moss_IMRN}), given by 
$$\big\{\varphi|_{T_j(F)} : \varphi \in \mathcal{V} ~(\text{resp.}~ \varphi \in \mathcal{V}_i )\big\}.$$
\noindent
We now state the following lemma, which plays a crucial role in the proof of Theorem~\ref{Rationality}.
\begin{lemma}\label{concluding_lemma_rational}
There exist monic polynomials $f_1,\ldots, f_n$ in \(R[X]\) with unit constant term such that, for any $j=1,\ldots,n$, the product $f_1(\rho_1(\varpi_F))\cdots f_j(\rho_j(\varpi_F))$ maps $\mathcal{V}^j$ into $\bigcap_{i\leq j} \mathcal{V}_i^j$.
\end{lemma}
\begin{proof}
To prove the lemma, we follow the idea of \cite[Lemma 4.3.1]{Moss_IMRN}. Let \( W \in \mathcal{V} \). We have to show that
there exist sufficiently large integers \( t_1, \dots, t_j \) (depending on $W$) such that
\[
f_1\bigl(\rho_1(\varpi_F)\bigr) \cdots f_j\bigl(\rho_j(\varpi_F)\bigr)\,
W(a_1,\dots,a_j) \;=\; 0
\]
whenever \( v(a_i) > t_i \) for some \( i \le j \).
These constants \( t_i \) are completely determined by \( W \). 

We argue the result by induction on \( n \). For \( n = 1 \), we have 
\(\displaystyle \bigcap_i \mathcal{V}_i = \mathcal{V}_1 \), so the claim follows immediately from Lemma \ref{finite_gen}.
Furthermore, the constant term of the annihilating polynomial is a unit since 
\(\rho_1(\varpi_F)\) is invertible.

Suppose the lemma is valid for $n-1$. Fix an element $W\in \mathcal{V}$. As the operator \( \rho_n(\varpi_F) \) acts invertibly and is integral over 
\(\End(\mathcal{V}/\mathcal{V}_n)\), it satisfies a monic polynomial 
\( f_n(X) \) with constant term equal to a unit.
As a result, there exists an integer \( t_n \) such that applying 
\( f_n(\rho_n(\varpi_F)) \) to \( W \) forces
\[
\bigl(f_n(\rho_n(\varpi))W\bigr)(a_1,\dots,a_n) = 0,
\]
whenever $\upsilon(a_n) > t_n$.

Now fix \( b \in F^\times \) and define a map
\(
\varphi_b : \prod_{i=1}^{n-1} F^\times 
\longrightarrow R\)
by
\[
\varphi_b(a_1,\dots,a_{n-1}) 
= \bigl(f_n(\rho_n(\varpi_F))W\bigr)(a_1,\dots,a_{n-1},b).
\]
Observe that \( \varphi_b \equiv 0 \) whenever \( v(b) > t_n \).
Applying the induction hypothesis to \( \mathcal{V}^{n-1} \), we obtain polynomials 
\( f_1,\dots,f_{n-1} \in R[X] \) with the desired properties: 
for any \( \varphi \in \mathcal{V}^{n-1} \), there exist sufficiently large integers 
\( t_1(\varphi),\dots,t_{n-1}(\varphi) \), depending on \( \varphi \), such that
\[
\bigl(f_1(\rho_1(\varpi_F)) \cdots f_{n-1}(\rho_{n-1}(\varpi_F))\,\varphi\bigr)
(a_1,\dots,a_{n-1}) = 0
\]
whenever at least one \( v(a_i) > N_i(\varphi) \) for \( 1 \le i \le n-1 \). 
Since the map \( \varphi_b \) is, by construction, the restriction of a product of Whittaker functions 
to \( T_{n-1}(F) \), we can apply the induction directly to \( \varphi_b \). 
This gives the existence of large integers \( t_1(b), \dots, t_{n-1}(b) \), depending on \( b \), such that
\[
\bigl(f_1(\rho_1(\varpi_F)) \cdots f_{n-1}(\rho_{n-1}(\varpi_F))\, 
f_n(\rho_n(\varpi_F))\, W\bigr)(a_1,\dots,a_{n-1},b) = 0
\]
whenever \( v(a_i) > t_i(b) \) for some $1\leq i \leq n-1$.

We aim to prove that the integers \( t_i \) can be chosen independently of \( b \).
Since \( \varphi_b \equiv 0 \) for \( v(b) > t_n \), and also vanishes for \( v(b) \ll 0 \) by Lemma~\ref{285},
the map \( \varphi_b \) is nonzero only when \( b \) lies in a compact subset of \( F^\times \).
Moreover, since \( f_n(\rho_n(\varpi_F))W \) is locally constant in each variable,
particularly in the last variable, there are only finitely many distinct functions
\( \varphi_b \) as \( b \) varies in this compact set.
It follows that the sets
\(\{ t_i(b) : b \in F^\times \}\)
are finite for each \( i \).
We can therefore take
$t_i = \max \{ t_i(b) : b \in F^\times \}.$
Consequently, 
\[
\bigl(f_1(\rho_1(\varpi_F)) \cdots f_n(\rho_n(\varpi_F))W\bigr)(a_1,\dots,a_n) = 0
\]
whenever \( v(a_i) > t_i \) for \( i = 1,\dots,n \), as required.
\end{proof}

We are now ready to deduce the rationality result, and the proof proceeds by adopting the approach of \cite[Theorem $3.2$]{Moss_IMRN}.
\begin{theorem}\label{Rationality}
Let $R$ be a Noetherian $\Lambda$-algebra, and let $V$ be a smooth $R[G_n(E)]$ module of Whittaker type. Let $S$ be the multiplicative subset of $R[X,X^{-1}]$ consisting of polynomials in $R[X,X^{-1}]$ whose first and last coefficients are units. Then $I(X,W,\varphi)\in S^{-1}(R[X,X^{-1}])$.
\end{theorem}
\begin{proof}
Let $W\in \mathbb{W}(V,\psi_E)$ and $\varphi\in C_c^\infty(F^n,R)$. Consider $W$ as an element of $\mathcal{V}$. Recall that $I(X,W,\varphi)$ is an element of $R[[X]][X^{-1}]$. For any integers $m_1,\dots,m_n$, we have the following identity
\begin{equation}\label{zeta_operator}
I(X,\rho_1(\varpi_F)^{m_1}\rho_2(\varpi_F)^{m_2}\cdots \rho_n(\varpi_F)^{m_n}(W),\varphi) = X^{m_1+2m_2+\dots+nm_n}I(X,W,\varphi)
\end{equation}
Using Lemma \ref{concluding_lemma_rational}, there exists monic polynomials $f_i\in R[X_i]$, $1\leq i\leq n$, with unit constant terms, such that the operator $f_1(\rho_1(\varpi_F))\cdots f_n(\rho_n(\varpi_F))$ maps $\mathcal{V}$ into $\bigcap_{i=1}^n \mathcal{V}_i$. Let $W_0 = f_1(\rho_1(\varpi_F))\cdots f_n(\rho_n(\varpi_F))(W)$. Then 
$$ I\big(X, f_1(\rho_1(\varpi_F))\cdots f_n(\rho_n(\varpi_F))(W),\varphi\big) = I(X,W_0,\varphi). $$
Since $W_0\in \bigcap_{i=1}^n\mathcal{V}_i$, there exists an integer $N>0$ such that $W_0(a) = 0$ for all $a\in T_n(F)$ with $v({\rm det}(a))> N$. This shows that $I(X,W_0,\varphi)\in R[X,X^{-1}]$. Now, consider the multivariate polynomial $f(X_1,\dots,x_n)\in R[X_1,\dots,X_n]$, defined as  
$$ f(X_1,\dots,X_n) = f_1(X_1)\cdots f_n(X_n). $$
Let $\widetilde{f}$ be the image of $f$ under the map
$$ R[X_1,\dots,X_n] \longrightarrow R[X] $$
$$ X_i \longmapsto X^i. $$
Using the identity (\ref{zeta_operator}), we get that $\widetilde{f}(X)\,I(X,W,\varphi) \in R[X,X^{-1}]$. Since $\widetilde{f}$ is monic with unit constant term, we have $\widetilde{f}\in S$. Hence, the lemma.
\end{proof}

\section{Functional equation}
In this section, we prove Theorem \ref{func_equ_intro}. The idea of the proof is similar to \cite[Appendix, Theorem(iii)]{Flicker_Asai}. We follow the same notation as in Section \ref{rationality}. Let $F$ be a $p$-adic field and let $E$ be a quadratic extension of $F$. Let $V$ be an $R[G_n(E)]$ module of Whittaker type, where $R$ is a Noetherian $\Lambda$-algebra. 
\subsection{}
Let $R_0 = S^{-1}(R[X,X^{-1}])$. For $r\in R_0^\times$, let $\chi_r:G_n(F)\rightarrow R_0^\times$ be the character defined by $\chi_r(g) = r^{-v({\rm det}g)}$ for $g\in G_n(F)$. We denote by $\mathcal{B}\big(\mathbb{W}(V,\psi_E),C_c^\infty(F^n,R);\chi_X\big)$ the space of bilinear maps $B:\mathbb{W}(V,\psi_E)\times C_c^\infty(F^n,R)\rightarrow R$ such that
$$ B(g.W, g.\varphi) = X^{-v({\rm det}g)} B(W,\varphi) $$
for all $g\in G_n(F)$, $W\in\mathbb{W}(V,\psi_E)$ and $\varphi\in C_c^\infty(F^n,R)$. Note that the formal series $I(X,W,\varphi)$, defined in Section \ref{Asai_Zeta}, certainly belongs to $\mathcal{B}\big(\mathbb{W}(V,\psi_E),C_c^\infty(F^n,R);\chi_X\big)$. We prove the following theorem.
\begin{theorem}\label{free_rank_1}
Let $R$ be a Noetherian $\Lambda$-algebra. Let $V$ be an $R[G_n(E)]$ module of Whittaker type. Then $\mathcal{B}\big(\mathbb{W}(V,\psi_E),C_c^\infty(F^n,R);\chi_X\big)$ is a free $R$-module of rank one.
\end{theorem}
To prove the above theorem, we recall some standard observations. For a locally profinite group $G$ and two smooth $R[G]$ modules $V_1,V_2$, let $\mathcal{B}_G(V_1,V_2)$ be the space of $R$-bilinear forms $B$ on $V_1\times V_2$ for which $B(gv_1,gv_2) = B(v_1,v_2)$ for all $g\in G$ and $(v_1,v_2)\in V_1\times V_2$. Note that the space of all functions $\varphi\in C_c^\infty(F^n,R)$ with $\varphi(0)=0$ is isomorphic to $C_c^\infty(F^n\setminus \{0\})$. Let $P_n(F)$ be the mirabolic subgroup of $G_n(F)$. The action of $G_n(F)$ on $F^n\setminus \{0\}$, defined by $g.\xi=\xi g$, for $g\in G_n(F)$ and $\xi\in F^n\setminus\{0\}$, induces an identification of $F^n\setminus \{0\}$ with the coset space $P_n(F)\backslash G_n(F)$. Then  
$$ C_c^\infty(F^n\setminus\{0\}) = C_c^\infty(P_n(F)\backslash G_n(F)) = \iota_{P_n(F)}^{G_n(F)}\delta_{P_n(F)}^{-1/2}, $$
where $\iota_{P_n(F)}^{G_n(F)}$ is the normalized compact induction functor. Here, $\delta_{P_n(F)}$ denotes the modulus character of $P_n(F)$, given by 
$$ \delta_{P_n(F)}
\begin{pmatrix}
g & x\\
0 & 1
\end{pmatrix} = q^{-\upsilon({\rm det}g)}, $$
for $g\in G_{n-1}(F)$ and $x\in F^{n-1}$.
Using Frobenius reciprocity and the fact that $\iota_{P_n(F)}^{G_n(F)}\delta_{P_n(F)}^{-1/2} \simeq I_{P_n(F)}^{G_n(F)}\delta_{P_n(F)}^{1/2}$, we have 
\begin{align*}
\mathcal{B}_{G_n(F)}\big(\mathbb{W}(V,\psi_E)\otimes \chi_{X^{-1}}, C_c^\infty(F^n\setminus\{0\}\big) 
&\simeq \mathcal{B}_{G_n(F)}\big(\mathbb{W}(V,\psi_E)\otimes \chi_{X^{-1}},\iota_{P_n(F)}^{G_n(F)}
\delta_{P_n(F)}^{-1/2}\big)\\
&\simeq \mathcal{B}_{P_n(F)}\big(\mathbb{W}(V,\psi_E)\otimes \chi_{X^{-1}},\delta_{P_n(F)}^{-1}\big)\\
&\simeq \mathcal{B}_{P_n(F)}\big(\mathbb{W}(V,\psi_E)\otimes \chi_{q^{-1}X^{-1}},\mathds{1}\big)
\end{align*}
where $\mathds{1}$ denotes the trivial $R[P_n(F)]$ module. The main step in the proof of Theorem \ref{free_rank_1} is to establish the following lemma.
\begin{lemma}\label{free_rank_intermediate}
The space $\mathcal{B}_{P_n(F)}\big(\mathbb{W}(V,\psi_E)\otimes \chi_{q^{-1}X^{-1}},\mathds{1}\big)$ is a free $R$-module of rank one.
\end{lemma}
\begin{proof}
Put $\mathcal{W} = \mathbb{W}(V,\psi_E)$. Consider the filtration of $R[P_n(E)]$ submodules associated with $\mathcal{W}$:
$$ \{0\}\subseteq \mathcal{W}_n\subseteq \cdots\subseteq \mathcal{W}_2\subseteq \mathcal{W}_1=\mathcal{W} $$
where $\tau_i = \mathcal{W}_i/\mathcal{W}_{i+1}\simeq (\Phi^+)^{i-1}\Psi^+(\mathcal{W}^{(i)})$. Following the arguments of the proof of \cite[Appendix, Main Lemma]{Flicker_Asai}, we get that the space
$\mathcal{B}_{P_n(F)}\big(\tau_i\otimes\chi_{q^{-1}X^{-1}},\mathds{1}\big)$ is trivial for $1\leq i\leq n-1$. Note that $\tau_n = {\rm ind}_{N_n(E)}^{P_n(E)}\psi_E$. The restriction of $\tau_n\otimes \chi_{q^{-1}X^{-1}}$ to $P_n(F)$ has a composition series consisting of ${\rm ind}_{N_n(F)}^{P_n(F)}\psi_E^{a}$, where $a$ varies over the coset space $T_n(E)/T_n(F)$ (recall that $T_n$ is the diagonal torus in $G_n$), and $\psi_E^a(u)=\psi_E(a^{-1}ua)$ for $u\in N_n(F)$ (\cite[Appendix, Lemma 9]{Flicker_Asai}). Since $\psi_E^a$ is non-trivial on $N_n(F)$, the space ${\rm Hom}_{N_n(F)}(\psi_E^a,\mathds{1})$ is trivial for all $a\notin T_n(F)$. If $a\in T_n(F)$, then ${\rm Hom}_{N_n(F)}(\psi_E,\mathds{1})\simeq R$ as $\psi_E$ is trivial on $N_n(F)$. This implies that $\mathcal{B}_{P_n(F)}(\tau_n\otimes\chi_{q^{-1}X^{-1}},\mathds{1}) = \mathcal{B}_{P_n(F)}(\tau_n,\mathds{1})$ is a free $R$-module of rank one. Hence, the lemma.
\end{proof}
We are now ready to prove Theorem \ref{free_rank_1}, by adapting the key ideas of \cite[Appendix, Proposition]{Flicker_Asai}.
\begin{proof}[Proof of Theorem \ref{free_rank_1}] 
For simplicity of notation, we put $\mathcal{C} =C_c^\infty(F^n,R)$ and $\mathcal{C}_0=C_c^\infty(F^n\setminus\{0\},R)$. Since the space $\mathcal{B}\big(\mathbb{W}(V,\psi_E), C_c^\infty(F^n,R);\chi_X\big)$ is non-zero, we need to show that two elements in the space are multiples of each other. We let $\phi, \phi'\in \mathcal{B}\big(\mathbb{W}(V,\psi_E), C_c^\infty(F^n,R);\chi_X\big)$. Note that
$$ \mathcal{B}\big(\mathbb{W}(V,\psi_E), C_c^\infty(F^n,R);\chi_X\big) = \mathcal{B}_{G_n(F)}\big(\mathbb{W}(V,\psi_E)\otimes\chi_{X^{-1}}, \mathcal{C}\big) $$
and
$$ \mathcal{B}_{P_n(F)}\big(\mathbb{W}(V,\psi_E)\otimes \chi_{q^{-1}X^{-1}}, \mathds{1}\big)
= \mathcal{B}_{G_n(F)}\big(\mathbb{W}(V,\psi_E)\otimes\chi_{X^{-1}}, \mathcal{C}_0\big) $$
By Lemma \ref{free_rank_intermediate}, there is an element $c\in R_0$ such that restriction of $\phi_0 = \phi-c\phi'$ to $\mathbb{W}(V,\psi_E)\otimes\chi_{X^{-1}} \times \mathcal{C}_0$ is zero. Also note that the map $\varphi\mapsto \varphi(0)$ is an isomorphism from $\mathcal{C}/\mathcal{C}_0$ onto $R_0$. Therefore the $G_n(F)$-equivariant bilinear form $\phi_0$ on $\mathbb{W}(V,\psi_E)\otimes\chi_{X^{-1}}\times \mathcal{C}$ is of the form
$$ \phi_0(W,\varphi) = \xi(W)\varphi(0). $$
where $\xi$ is a $G_n(F)$-equivariant linear form on $\mathbb{W}(V,\psi_E)\otimes\chi_{X^{-1}}$. For all $z\in F^\times$, we have $\xi(z.W) =\xi(W)$, which implies that $(\omega_V(z)X^{n\upsilon(z)} - 1)h(W) =0$ in $R_0$. Since the polynomial $w_V(z)X^{n\upsilon(z)}-1$, being an element of the multiplicative set $S$, is a unit in $R_0$, we have $\xi(W) =0$, and hence $\phi_0(W,\varphi) =0$ for all $W\in\mathbb{W}(V,\psi_E)\otimes\chi_{X^{-1}}$ and $\varphi\in \mathcal{C}$. Thus, we get $\phi = c\phi'$. This completes the proof.
\end{proof}
As a consequence of Theorem \ref{free_rank_1}, we get the following functional equation.
\begin{theorem}\label{functional_e to showqu}
Let $E$ be a quadratic extension of a $p$-adic field $F$. Let $R$ be a Noetherian $\Lambda$-algebra, and let $V$ be an $R[G_n(E)]$ module of Whittaker type. There exists a unique element $\gamma_{As}(X,V,\psi_E)\in S^{-1}(R[X,X^{-1}])$ such that 
$$ I(q^{-1}X^{-1}, \widetilde{W},
\widehat{\varphi}) = 
\omega_{V}(-1)^{n-1}\,\gamma_{As}(X,V,\psi_E)\, I(X,W,\varphi), $$
for all $W\in \mathbb{W}
(V,\psi_E)$ and $\varphi
\in C_c^\infty(F^n,R)$.
\end{theorem}
\begin{proof}
Note that both the formal series $I(X,W,\varphi)$ and $I(q^{-1}X^{-1},\widetilde{W},\widehat{\varphi})$ belong to the space $\mathcal{B}\big(\mathbb{W}(V,\psi_E), C_c^\infty(F^n,R);\chi_X\big)$, which is a free $R$-module of rank one by Theorem \ref{free_rank_1}. Hence, the theorem.
\end{proof}
Let $R'$ be a Noetherian $\Lambda$-algebra, and let $f: R\rightarrow R'$ be a morphism of $\Lambda$-algebras. Then $V\otimes_{R,f} R'$ is also an $R'[G_n(E)]$ module of Whittaker type. Then we have the following lemma, which shows that the construction of the formal series $I(X,W,\varphi)$ is compatible with the base change of the coefficient rings. 
\begin{lemma}
Let $T$ denote the ring homomorphism $S^{-1}(R[X,X^{-1}])\rightarrow S^{-1}(R'[X,X^{-1}])$, induced by $f$. Then 
$$ T\big(I(X,W,\varphi)\big) = I(X, f\circ W, f\circ \varphi) $$
and 
$$ T\big(I(X,\widetilde{W},
\widehat{\varphi})\big) = I(X, \widetilde{f\circ W}, \widehat{f\circ \varphi)}, $$
for all $W\in \mathbb{W}(V,\psi_E)$ and $\varphi\in C_c^\infty(F^n, R)$.
\end{lemma}
\begin{proof}
Note that 
$\mathbb{W}(V\otimes_{R,f} R',\psi_E) = \mathbb{W}(V,\psi_E)\otimes_{R,f} R'$.
The lemma then follows from the definition of $I(X,W,\varphi)$.
\end{proof}
As a corollary, we have the following result.
\begin{corollary}\label{gamma_comp}
If $\gamma_{As}(X,V,\psi_E)$ satisfies the functional equation as in Theorem \ref{functional_e to showqu}, then $T(\gamma_{As}(X,V,\psi_E))$ satisfies a functional equation for $V\otimes_{R,f} R'$. Moreover, by the uniqueness of such an element as in Theorem \ref{functional_e to showqu}, we have
$$ T\big(\gamma_{As}(X,V,\psi_E)\big) = \gamma_{As}(X,V\otimes_{R,f} R',\psi_E). $$
\end{corollary}
The element $\gamma_{As}(X,V,\psi_E)$ is called the {\it Asai gamma factor} associated with $V$. The above corollary shows that the Asai $\gamma$-factor is compatible with the base change of the coefficient rings. 

\begin{remark}\normalfont
In particular, when $E=F\times F$, then we obtain Rankin--Selberg functional equation in families for a pair of $R[G_n(F)]$ modules $(V,V')$ of Whittaker type. Infact, the same proof works for $R = A \otimes_{\Lambda} B$, where $A$ and $B$ are Noetherian $\Lambda$-algebras--which completes the Rankin--Selberg theory of \cite[Corollary 3.10]{Moss_Nadir_Kirillov} for $m=n$.
\end{remark}

\section{Distinguished representations in families}
We begin this section by introducing the notion of Whittaker distinguished representations in families. Then we prove a necessary condition for modules of Whittaker type to be Whittaker distinguished, in terms of their Asai $\gamma$-factors.
\subsection{}
Let $F$ be a $p$-adic field, and let $E$ be a quadratic extension of $F$. Let $R$ be a Noetherian $\Lambda$-algebra. A smooth $R[G_n(E)]$ module $V$ is said to be $G_n(F)$-distinguished if there exists a non-zero $R$-linear map $\lambda: V\rightarrow R$ such that $\lambda(gv) = \lambda(v)$ for all $g\in G_n(F)$ and $v\in V$. For our purposes, we define distinguished representations for Whittaker type families as follows.  
\begin{definition}\normalfont
A smooth $R[G_n(E)]$ module $V$ of Whittaker type is said to be Whittaker $G_n(F)$-distinguished if there exists a non-zero $G_n(F)$-equivariant linear map $\lambda:\mathbb{W}(V,\psi_E)\rightarrow R$. This is equivalent to saying 
$${\rm Hom}_{R[G_n(F)]}(\mathbb{W}(V,\psi_E),\mathds{1}) \not= \{0\}.$$
\end{definition}
Note that a $G_n(F)$ equivariant map $\lambda: \mathbb{W}(V,\psi_E)\rightarrow R$ induces a $G_n(F)$ equivariant map on $V$ by precomposing $\lambda$ with the canonical surjection $V\rightarrow \mathbb{W}(V,\psi_E)$. Therefore, an $R[G_n(E)]$ module of Whittaker type that is Whittaker $G_n(F)$-distinguished, is also $G_n(F)$-distinguished. But an $R[G_n(E)]$ module of Whittaker type that is $G_n(F)$-distinguished, may not be Whittaker $G_n(F)$-distinguished. If $R$ is a field and $V$ is an irreducible $R[G_n(E)]$ module of Whittaker type (in which case $V$ is generic), then the surjection $V\rightarrow \mathbb{W}(V,\psi_E)$ is an isomorphism, in which case if $V$ is $G_n(F)$-distinguished, then it is Whittaker $G_n(F)$-distinguished. Therefore, the notion of the Whittaker distinguished representations is stronger than the usual notion of distinguished representations for Whittaker type families.
Note that if $V$ is Whittaker $G_n(F)$-distinguished, then so is $V^\vee$.

When $R=\mathbb{C}$, there is a characterization of discrete series complex representations of $G_n(E)$ that are $G_n(F)$-distinguished, in terms of the poles of their Asai $L$-functions, due to \cite[Theorem 4]{Kable_Asai} and \cite[Corollary 1.5]{Anand_Kable_Asai}. The result is as follows.
\begin{theorem}\label{Kable_Asai_thm}
Let $(\pi,V)$ be a discrete series representation of $G_n(E)$. Then $\pi$ is $G_n(F)$-distinguished if and only if the Asai $L$-function $L_{As}(\pi,s)$ has a pole at $s=0$.
\end{theorem}
Recall that the notion of $L$-factors in families is not well-defined because it is difficult to construct $L$-factors in a way compatible with the base change of coefficient rings (see \cite[Section 1, P. 1790]{Moss_local_constants}). On the other hand, the $\gamma$-factors for modules of Whittaker type are well-defined and they are compatible with the base change of the coefficient rings. So, it is natural to ask for a similar kind of result as Theorem \ref{Kable_Asai_thm} for $R[G_n(E)]$ modules of Whittaker type that are Whittaker $G_n(F)$-distinguished, in terms of Asai $\gamma$-factors. For that, we need to recall and establish some facts.
\subsection{}
Let $V$ be an $R[G_n(E)]$ module of Whittaker type. Assume that $V$ is cuspidal, i.e., the Jacquet module $J_P(V)$ is trivial for all proper parabolic subgroups $P$ of $G_n(E)$. Consider the filtration of $R[P_n(E)]$ submodules:
$$ \{0\}\subseteq V_n \subseteq \cdots\subseteq V_2\subseteq V_1=V $$
with $V_i/V_{i+1} \simeq (\Phi^+)^{i-1}\Psi^+(V^{(i)})$ for $1\leq i\leq n$ and $V_n = {\rm ind}_{N_n(E)}^{P_n(E)}\psi_E$. For each $0\leq i\leq n$, let $M(i)$ be the Levi subgroup $G_{n-i}(E)\times G_{i}(E)$. Then the $i$-th derivative $V^{(i)}$ is the composition $(J_{M(i)}(V))^{(0,i)}$ (\cite[Proposition 6.7]{BZ_2}, \cite[Chapter III, Section 1.8]{Vigneras_modl_book}), where $J_{M(i)} : \mathcal{M}(G_n(E))\rightarrow \mathcal{M}(M(i))$ is the Jacquet functor and $(-)^{(0,i)}:\mathcal{M}_R(M(i))\rightarrow \mathcal{M}_R(G_{n-i}(E))$ is the partial derivative functor (see \cite[Section 2.3]{Moss_local_constants} for the definition). Since $V$ is cuspidal, we have $V^{(i)} =0$ for $1\leq i\leq n-1$, and hence it follows from the above filtration that $\mathbb{W}(V,\psi_E)|_{P_n(E)} = {\rm ind}_{N_n(E)}^{P_n(E)}\psi_E$. We define a $P_n(F)$-equivariant linear form $\lambda$ on the Whittaker space $\mathbb{W}(V,\psi_E)$ by 
\begin{equation}\label{linear_form}
\lambda(W) :=\int_{N_n(F)\backslash P_n(F)} W(p)\, dp,
\end{equation}
where $dp$ is a right $P_n(F)$-equivariant Haar measure on $N_n(F)\backslash P_n(F)$. Since $V$ is cuspidal, the above integral makes sense. We now fix a congruence subgroup $K'$ of ${\rm GL}_n(\mathfrak{o}_E)$, and let $\varphi\in {\rm ind}_{N_n(E)}^{P_n(E)}\psi_E$ be such that it is supported on $N_n(E)(K'\cap P_n(E))$ and right invariant under $K'\cap P_n(E)$. Then there exists $W\in\mathbb{W}(V,\psi_E)$ such that $W|_{P_n(E)} = \varphi$ and $\lambda(W) \ne 0$. Thus, $\lambda$ is a non-zero $P_n(F)$-equivariant $R$-linear form on $\mathbb{W}(V,\psi_E)$.  We prove the following theorem, which says that the linear form $\lambda$, defined by (\ref{linear_form}), is the only $P_n(F)$-equivariant linear form on $\mathbb{W}(V,\psi_E)$ up to a scalar.
\begin{theorem}\label{Hom_free}
Let $R$ be a Noetherian $\Lambda$-algebra, and let $V$ be a cuspidal $R[G_n(E)]$ module of Whittaker type. Then 
$$ {\rm Hom}_{R[P_n(F)]}\big(\mathbb{W}(V,\psi_E),\mathds{1}\big) $$ 
is a free $R$-module of rank one. 
\end{theorem}
\begin{proof}
Let $\mathcal{K}(\psi_E)$ denote the $R[P_n(E)]$ module ${\rm ind}_{N_n(E)}^{P_n(E)}\psi_E$. The Whittaker model 
\(\mathbb W(V,\psi_E)\), when viewed as a representation of \(P_n(F)\), carries the ordinary Bernstein--Zelevinsky filtration. 
The \(i\)-th graded piece of this filtration may be taken to be isomorphic to
\[
(\Phi^+)^{\,i-1}\,\Psi^+\!\big(\mathbb W(V,\psi_E)^{(i)}\big),
\quad 1 \le i \le n,
\]
where \(\Phi^\pm\) and \(\Psi^\pm\) are the usual mirabolic functors and 
\(\mathbb W(V,\psi_E)^{(i)}\) denotes the \(i\)-th Bernstein--Zelevinsky derivative. 
To compute the space of \(P_n(F)\)-equivariant maps into the trivial module, 
we apply the functor \(\mathrm{Hom}_{P_n(F)}(\,\cdot\, , \mathds{1})\) to the short exact sequences 
of the filtration. This reduces the problem to determining the Hom-space on 
each graded piece. Since $V$ is cuspidal, $\mathbb{W}(V,\psi_E)|_{P_n(E)}=\mathcal{K}(\psi_E)$. Moreover, the $i$-th derivatives $\mathcal{K}(\psi_E)^{(i)}$ are trivial for $i<n$ and $\mathcal{K}(\psi_E)^{(n)}$ is a free $R$-module of rank one. Therefore, we have 
$$ {\rm Hom}_{R[P_n(F)]}\big(\mathbb{W}(V,\psi_E),\mathds{1}\big)\, \simeq\, {\rm Hom}_{R[P_n(F)]}\big((\Phi^+)^{n-1}\Psi^+(\mathds{1}),\mathds{1}\big). $$
Note that $(\Phi^+)^{n-1}\Psi^+((\mathds{1}) = \mathcal{K}(\psi_E)$ and hence, the latter Hom-space is isomorphic to $R$ (see the proof of Lemma \ref{free_rank_intermediate}). This proves the theorem.
\end{proof}
\begin{corollary}\label{distinguished_useful}
Let $V$ be a cuspidal $R[G_n(E)]$ module of Whittaker type such that $V$ is Whittaker $G_n(F)$-distinguished. Then the linear form $\lambda$, defined by (\ref{linear_form}), induces a non-zero $G_n(F)$-equivariant $R$-linear form on $\mathbb{W}(V,\psi_E)$.
\end{corollary}
\begin{proof}
Since a $G_n(F)$-equivariant linear form is, in particular, $P_n(F)$-equivariant, the corollary follows from Theorem \ref{Hom_free}.
\end{proof}
\subsection{}
In this subsection, we recall some general facts. Let $R$ be a commutative ring with unity such that $p\in R^\times$. Let $G$ be a locally profinite group which admits a compact open subgroup of pro-order invertible in $R$. By \cite[Chapter I, Section 2.4]{Vigneras_modl_book}, for each compact open subgroup $K$ of $G$ of pro-order invertible in $R$, there exists a unique left Haar measure $\mu$ such that $\mu(K) =1$. Let $H$ be a closed subgroup of $G$. Let $\delta_G:G\rightarrow R^\times$ (resp. $\delta_H$) be the modulus character of $G$ (resp. $H)$, and we let $\delta= \delta_G^{-1}|_H \,\delta_H$. Let $dh$ be a right Haar measure on $H$, and let $d_{H\backslash G} g$ be a $\delta$-quasi-invariant quotient measure on $H\backslash G$. Then we can define a Haar measure $dg$ on $G$ such that 
$$ \int_{G} f(g)\,dg = \int_{H\backslash G} \bigg(\int_H f(hg)\,\delta^{-1}(h)\,dh\bigg)\,d_{H\backslash G}(g), $$
for all $R$-valued functions $f$ which are locally constant and compactly supported (see \cite[Section 2.2]{Rankin_Selberg_mod_l}). As a consequence of this identity, we have the following crucial lemma, which will be used in the main theorem.
\begin{lemma}\label{int_decom}
Let $G$ be a locally profinite group. Let $H_1\subseteq H_2$ be two closed subgroups of $G$ such that the quotient $H_1\backslash H_2$ admits a right $H_2$-invariant Haar measure $\nu_2$. Let $\mu_1$ and $\mu_2$ be the right $G$-invariant measures on the quotient spaces $H_1\backslash G$ and $H_2\backslash G$, respectively. Then, for suitable normalization of Haar measures, we have
$$ \int_{H_1\backslash G} f(g)\,d\mu_1(g) = 
\int_{H_2\backslash G}\bigg(\int_{H_1\backslash H_2} f(hg)\, d\nu_2(h)\bigg)\,d\mu_2(g), $$
for $R$-valued functions $f$ locally constant with compact support modulo $H_1$. 
\end{lemma}
We prove the following theorem.
\begin{theorem}\label{main_theorem}
Let $R$ be Noetherian $\Lambda$-algebra. Let $V$ be a cuspidal $R[G_n(E)]$ module of Whittaker type such that $V$ is Whittaker $G_n(F)$-distinguished. Then the Asai $\gamma$-factor $\gamma_{As}(X,V,\psi_E)$ vanishes at $X=1$.
\end{theorem}
\begin{proof}
Let $\mu_1$  and $\mu_2$ be the right $G_n(F)$-equivariant Haar measure on the quotient spaces $N_n(F)\backslash G_n(F)$ and $P_n(F)\backslash G_n(F)$, respectively. Let $dp$ be a right $P_n(F)$-equivariant measure on $N_n(F)\backslash P_n(F)$. Recall that
$$ I(q^{-1}X^{-1},\widetilde{W},\widehat{\varphi}) 
= \sum_{r\in\mathbb{Z}} \bigg(\int_{N_n(F)\backslash G_n(F)^r} \widetilde{W}(g)\,\widehat{\varphi}(\eta g)\,d\mu_1(g)\bigg) X^r, $$
for $W\in \mathbb{W}(V,\psi_E)$ and $\varphi\in C_c^\infty(F^n,R)$. Putting $X=1$ in the above identity and using Lemma \ref{int_decom}, we get
\begin{align*}
I(q^{-1},\widetilde{W},\widehat{\varphi}) 
&=
\int_{P_n(F)\backslash G_n(F)}\bigg(\int_{N_n(F)\backslash P_n(F)}\widetilde{W}(pg)\,\widehat{\varphi}(\eta pg)\,dp\bigg)\,d\mu_2(g)\\
&= \int_{P_n(F)\backslash G_n(F)}\widehat{\varphi}(\eta g)\bigg(\int_{N_n(F)\backslash P_n(F)}(g.\widetilde{W})(p)\,dp\bigg)\,d\mu_2(g).
\end{align*}
Since $V^\vee$ is Whittaker $G_n(F)$-distinguished, it follows from Corollary \ref{distinguished_useful} that the $P_n(F)$-equivariant map $\widetilde{W}\mapsto \int_{N_n(F)\backslash P_n(F)} \widetilde{W}(p)\,dp$ gives a $G_n(F)$-equivariant map on $\mathbb{W}(V^\vee,\psi_E^{-1})$. Then we have
\begin{align*}
I(q^{-1},\widetilde{W},\widehat{\varphi}) 
&=\int_{P_n(F)\backslash G_n(F)}\widehat{\varphi}(\eta g)\bigg(\int_{N_n(F)\backslash P_n(F)}(g.\widetilde{W})(p)\,dp\bigg)\,d\mu_2(g)\\
&=\bigg(\int_{P_n(F)\backslash G_n(F)}\widehat{\varphi}(\eta g)\,d\mu_2(g)\bigg)\bigg(\int_{N_n(F)\backslash P_n(F)}\widetilde{W}(p)\,dp\bigg)\\
&= \varphi(0) \int_{N_n(F)\backslash P_n(F)}\widetilde{W}(p)\,dp,
\end{align*}
where the first integral equals $\varphi(0)$ by Fourier inversion formula. From the functional equation (Theorem \ref{functional_e to showqu}), we get
\begin{equation}\label{func_equ_1}
I(q^{-1},\widetilde{W},\widehat{\varphi}) = \omega_{V}(-1)^{n-1}\, \gamma_{As}(1,V,\psi_E)\,I(1,W,\varphi)
\end{equation}
Now, one can choose $W_0\in \mathbb{W}(V,\psi_E)$ and $\varphi_0\in C_c^\infty(F^n,R)$ in such a way that $\varphi_0(0) = 0$ and $I(1,W_0,\varphi_0) = 1$ (see the proof of \cite[Theorem 1.4]{Anand_Kable_Asai}). Then it follows from the identity (\ref{func_equ_1}) that
$$ \gamma_{As}(1,V,\psi_E) = 0. $$
Hence, the theorem.
\end{proof}
\begin{remark}\normalfont
It is also interesting to understand the converse part of Theorem \ref{main_theorem}, and one may consider investigating it as future work. 
\end{remark}
\textbf{Acknowledgments.} The authors would like to thank U. K. Anandavardhanan for suggesting the problem and for several helpful discussions during the work. The authors thank Nadir Matringe for his helpful comments and suggestions, and to Dipendra Prasad for his encouragement and interest in this work. The authors also thank the Indian
Institute of Technology Bombay for supporting this research through the Institute
Postdoctoral Fellowship.
\bibliographystyle{amsalpha}
\bibliography{Asai_families}
\vspace{0.3 cm}
Sabyasachi Dhar,\\
\texttt{mathsabya93@gmail.com},
\texttt{sabya@math.iitb.ac.in}
\vspace{0.2 mm}\\
Hariom Sharma, \\
\texttt{hariomshrma97@gmail.com}, \texttt{hariom@math.iitb.ac.in}.\\
Department of Mathematics, Indian
Institute of Technology Bombay, Mumbai,  Maharashtra-400076, India.
\end{document}